\documentclass[reqno,a4paper,twoside]{amsart}
\title{Tensor generators on schemes and stacks}
\author{Philipp Gross}
\date{\today}

\usepackage[pdfstartview={FitH},
plainpages={false},
pdfpagelabels={true},
pdfpagemode={UseOutlines},
bookmarksopen={false},
colorlinks={false},
linkcolor={blue},
citecolor={green},
pdftitle={Tensor generators on schemes and stacks},
pdfauthor={Philipp Gross},
pdfsubject={Algebraic Geometry},
pdfkeywords={Resolution property, locally free sheaf, vector bundle, generating family, tensor generator, algebraic space, algebraic stack, quotient stack, AF-scheme}]{hyperref}%
\usepackage[T1]{fontenc}  

\usepackage{srcltx}  % Jump between dvi and tex
\usepackage{amsmath, amsthm, amssymb, amsfonts}
\usepackage{mathrsfs} % RSFS fonts
\usepackage{fixltx2e}[2005/12/01]
\usepackage[all]{xy} % Diagrams

\usepackage{enumitem}
\setlist{itemsep=0.5em}

\usepackage{natbib}
\bibpunct{[}{]}{,}{n}{}{;}

\newtheoremstyle{dtheoremnopar}{3 mm}{1 mm}{\itshape}{}{\bfseries}{.}{ }
{\thmname{#1}\thmnumber{ #2}\thmnote{ \mdseries(#3)\bfseries}}

\theoremstyle{dtheoremnopar}
\newcounter{theoremx}
\newtheorem{theoremalpha}[theoremx]{Theorem}

\newtheoremstyle{dtheorem}{3 mm}{1 mm}{\itshape}{}{\bfseries}{.}{ }
  {\thmnumber{(#2) }\thmname{#1}\thmnote{ \mdseries(#3)\bfseries}}
\newtheoremstyle{ddef}{3 mm}{1 mm}{\normalfont}{}{\bfseries}{.}{ }
  {\thmnumber{(#2) }\thmname{#1}\thmnote{ \mdseries(#3)\bfseries}}
\newtheoremstyle{dremark}{3 mm}{1 mm}{\normalfont}{}{\itshape}{}{ }
  {\thmnumber{\upshape\bfseries(#2) }\itshape\mdseries\thmname{#1}.\thmnote{\;\mdseries #3 ---}}

\theoremstyle{dtheorem}
\newtheorem{thm}{Theorem}[section]

\numberwithin{equation}{thm}
\makeatletter
\@addtoreset{thm}{section}
\makeatother

\newtheorem{prop}[thm]{Proposition}
\newtheorem{lem}[thm]{Lemma}
\newtheorem{cor}[thm]{Corollary}
\theoremstyle{ddef}
\newtheorem{ex}[thm]{Example}
\newtheorem{defi}[thm]{Definition}
\theoremstyle{dremark}
\newtheorem{rem}[thm]{Remark}

\newcommand{\msheaf}[1]{\mathcal{#1}}		
\newcommand{\shC}{\msheaf{C}} 
\newcommand{\shE}{\msheaf{E}}\newcommand{\shF}{\msheaf{F}}
\newcommand{\shG}{\msheaf{G}}\newcommand{\shI}{\msheaf{I}}
\newcommand{\shL}{\msheaf{L}}
\newcommand{\shM}{\msheaf{M}}\newcommand{\shN}{\msheaf{N}}\newcommand{\shO}{\msheaf{O}}

\newcommand{\shV}{\msheaf{V}}

\newcommand{\shHom}{\msheaf{H}om}
\newcommand{\shIsom}{\msheaf{I}som}

\newcommand{\bbA}{\mathbb{A}}

\newcommand{\bbG}{\mathbb{G}}

\newcommand{\bbN}{\mathbb{N}}
\newcommand{\bbQ}{\mathbb{Q}}

\newcommand{\bbZ}{\mathbb{Z}}

\newcommand{\Qco}{\operatorname{QCoh}}

\newcommand{\coloneq}{\ensuremath{:=}}

\newcommand{\AF}{\text{AF}}

\newcommand{\pr}{\text{pr}}
				
\newcommand{\red}{\text{red}}

\newcommand{\fppf}{\text{fppf}}				
\newcommand{\fpqc}{\text{fpqc}}

\newcommand{\et}{\text{\'et}}

\DeclareMathOperator{\coker}{\operatorname{coker}}

\DeclareMathOperator{\id}{\operatorname{id}}

\DeclareMathOperator{\GL}{\operatorname{GL}}

\newcommand{\Gm}{\bbG_m}

\DeclareMathOperator{\Mod}{\operatorname{Mod}}
\DeclareMathOperator{\rk}{\operatorname{rank}}
\DeclareMathOperator{\Sym}{\operatorname{Sym}}

\DeclareMathOperator{\Aut}{\operatorname{Aut}}

\DeclareMathOperator{\Coho}{\operatorname{H}}

\DeclareMathOperator{\Ext}{\operatorname{Ext}}
\DeclareMathOperator{\Hom}{\operatorname{Hom}}

\DeclareMathOperator{\Ann}{\operatorname{Ann}}

\DeclareMathOperator{\norm}{\operatorname{N}}

\DeclareMathOperator{\Spec}{\operatorname{Spec}}

\newcommand{\clsp}{\ensuremath{B}}

\newcommand{\pts}[1]{\left | #1 \right |}

\newcommand{\pf}[1]{ {#1}_* }
\newcommand{\pb}[1]{ {#1}^* }

\newcommand{\us}[1]{ {#1}^!}

\newcommand{\newdef}[1]{\emph{#1}}

\begin{document}

\begin{abstract}
We show that an algebraic stack with affine stabilizer groups satisfies the resolution property if and only if it is a quotient of a quasi-affine scheme by the action of the general linear group, or equivalently, if there exists a vector bundle whose associated frame bundle has quasi-affine total space. This generalizes a result of B. Totaro to non-normal and non-noetherian schemes and algebraic stacks. 
Also, we show that the vector bundle induces such a quotient structure if and only if it is a tensor generator in the category of quasi-coherent sheaves.
\end{abstract}
\maketitle

%%%%%%%%%%%%%%%%%%%%%%%%%%%%%%%%%%%%%%%%%%%%%%%%%%%%%%%%%%%%%%%%%%%%%%%%%%%%%%%%%%
\section*{Introduction}
\label{SEC:intro}
%%%%%%%%%%%%%%%%%%%%%%%%%%%%%%%%%%%%%%%%%%%%%%%%%%%%%%%%%%%%%%%%%%%%%%%%%%%%%%%%%%

When $X$ is a scheme, or an algebraic stack, such that every quasi-coherent sheaf of finite type  $\shF$ is a quotient of a locally free sheaf $\shG$, then one says that $X$ satisfies the \emph{resolution property}, and this has important consequences for $K$-theory. Whereas this condition holds for quasi-projective or $\bbQ$-locally factorial schemes with affine diagonal, it has been difficult to verify for general schemes that do not possess enough invertible sheaves. It is easy to construct counterexamples when the diagonal is not affine, but even the case of toric separated threefolds is still open. 

A related question deals with global quotient stack presentations.
When $X$ is a quotient stack $[U/G]$, for some algebraic space  $U$  acted on by a group scheme $G$, then the geometry of $X$ is equivalent to the $G$-equivariant geometry of $U$. Therefore, it is a fundamental question, whether such a global quotient stack presentation exists.

One says that $X$  is \emph{basic} if $U$ is a quasi-affine scheme and  $G=\GL_{n}$ \cite[Def. 2.1]{Ryd15:noetherianapprox}, or equivalently, if $X$ is endowed with a vector bundle $\shV$ of rank $n$, whose associated frame bundle $U\to X$ has quasi-affine total space. Such a stack has affine diagonal, and at each point $x$ of $X$, the stabilizer group $\Aut(x)$ is affine and acts faithfully on $\shV_x$. But not every algebraic stack with affine diagonal is basic. For example, the $\Gm$-gerbe associated with a non-torsion element of the cohomological Brauer group is not a global quotient stack.

It is a classical result by Thomason that a noetherian basic stack $X$ satisfies the resolution property \cite{Tho87:MR893468}. The converse was proven by Totaro  under the hypothesis that $X$ is a normal noetherian algebraic stack, and that the stabilizer groups are affine at closed points \cite[Thm.~1.1]{Tot04:MR2108211}.  He observed that a careful polynomial combination of enough locally free sheaves produces a locally free sheaf whose associated frame bundle is eventually representable by a quasi-affine scheme. 

The goal of this paper is to prove the equivalence of the resolution property and being basic in large generality, by removing the normality and noetherian assumption. 

Moreover, we show that for a basic stack $X$, the distinguished vector bundle $\shV$ is a \emph{tensor generator} in the sense that it generates ``enough'' locally free sheaves by the process of taking direct sums, duals, tensor products and certain locally split subsheaves. This generalizes the well-known fact that linear representations of algebraic group schemes can be reconstructed from a faithful representation.

In more detail, the main result is the following:

\begin{theoremalpha}\label{THM:thmA}
Let $X$ be a quasi-compact and quasi-separated algebraic stack (over $\bbZ$). Then the following are equivalent.
\begin{enumerate}
	\item \label{THM:IT:thmA:resprop} $X$ has affine stabilizer groups at closed points, and satisfies the resolution property.
	\item \label{THM:IT:thmA:tengen} $X$ has affine stabilizer groups at closed points, and admits a tensor generator $\shV$ of rank $n$: Every quasi-coherent sheaf $\shF$ of finite type is a quotient of a subsheaf $\shG \subset P(\shV, \shV^\vee)$ for some polynomial $P \in \bbN[r,s]$, such that $\shG$ is a direct summand, when restricted to the frame bundle $\shIsom_X(\shO_X^n, \shV)$.
	\item \label{THM:IT:thmA:basic} $X$ is basic, i.e $X=[U/\GL_n]$, where $U$ is a quasi-affine scheme with an action of $\GL_n$, $n \geq 0$. In particular, $X$ has affine diagonal.
\end{enumerate}
Moreover, if $\shV$ is a tensor generator, then $U$ in \ref{THM:IT:thmA:basic} can be taken as the total space of the frame bundle associated to $\shV$ and vice versa. Also, if $\shV$ has linearly reductive structure group (for example, if $\shV$ is a direct sum of line bundles, or if $X$ is of characteristic $0$), then in \ref{THM:IT:thmA:tengen}, the sheaf $\shF$ is a quotient of some $P(\shV, \shV^\vee)$.
\end{theoremalpha}

As an application of the linearly reductive case, we derive that an algebraic stack $X$ has a generating family of invertible sheaves and affine stabilizer groups, if and only if $X=[U/\Gm^n]$ for some quasi-affine scheme $U$ (Corollary \ref{COR:split_tensor_generator}). This was proven by Hausen for integral schemes of finite type over an algebraically closed field \cite[Thm. 1]{Hau02:MR1900763}.

Note that when $n=0$, the theorem implies that an algebraic stack $X$ with affine stabilizers at closed points is representable by a quasi-affine scheme if and only if every quasi-coherent sheaf is globally generated. On the other hand, if $X=\clsp \GL_n$ and $\shV$ corresponds to the standard representation of $\GL_n$, we recover the fact that every linear representation is a subquotient of a polynomial expression in $\shV$ and $\shV^\vee$. The general case follows by constructing a quasi-affine morphism $X \to \clsp \GL_{n}$.

In the process of verifying the aforementioned representability criterion we prove that the $AF$-property (i.e.~every finite set of points admits an affine open neighborhood, also known as the Chevalley-Kleiman property) descends along finite surjections in the category of algebraic spaces:

\begin{theoremalpha}\label{THM:AF_descends_simple}
Let $X$ be a quasi-compact and quasi-separated algebraic space and let $f \colon Y \to X$ be a finite surjective morphism. Then $X$ is an $\AF$-scheme if and only if $Y$ is an $\AF$-scheme.
\end{theoremalpha}
Kollar also independently proved this in the excellent case \cite[Cor. 48]{Kol11}, so our result is just a mild generalization.

\smallskip

This article is largely based on my thesis from 2010. The recent improvements of approximating general algebraic stacks \cite{Ryd15:noetherianapprox} made it possible to remove many technical assumptions.

\subsection*{Outline}
In section \ref{gensh:SEC} we define relatively generating families of finitely presented quasi-coherent sheaves with respect to a morphism of algebraic stacks and give a thorough discussion of the  basic permanence properties (Proposition \ref{PROP:permanence}). 

Section \ref{SEC:pinching} deals with pinching $\AF$-schemes, and there we prove Theorem \ref{THM:AF_descends_simple} and generalizations thereof.

We characterize in section \ref{SEC:quasi-affiness}  quasi-affine morphisms in terms of the counit $\pb f \pf f \shF \to \shF$. The proof involves several reduction steps from algebraic stacks, algebraic spaces to the well-known case of schemes. 

Section \ref{SEC:resprop} deals with relatively generating families of vector bundles, which enables us to treat the resolution property for morphisms $X \to Y$ of algebraic stacks. For the reader's convenience, we recall the main classes of algebraic stacks where the resolution property is known to be true.

Finally, in section \ref{SEC:tensorgen}, we give the proof of Theorem \ref{THM:thmA}. 

\subsection*{Acknowledgments}
I would like to thank Stefan Schr\"oer for the supervision and Burt Totaro for the review of my thesis. I am grateful to David Rydh for his kind support, for the numerous helpful suggestions, and for pointing out errors and inaccuracies.
This work has benefited greatly from conversations with Sam Payne, Andrew Kresch, Gregg Zuckerman, Jarod Alper, Holger Partsch, Felix Sch\"uller and Christian Liedtke.

\subsection*{Conventions and notations}

For algebraic stacks we follow the conventions in \cite{LM00:MR1771927} except that we do not require that the diagonal of an algebraic stack is separated, following \cite{Stacks-Project}. All our stacks will however be quasi-separated, that is, have quasi-compact and quasi-separated diagonal.
A vector bundle is a locally free sheaf of finite type, equivalently a flat and finitely presented quasi-coherent sheaf.

%%%%%%%%%%%%%%%%%%%%%%%%%%%%%%%%%%%%%%%%%%%%%%%%%%%%%%%%%%%%%%%%%%%%%%%%%%%%%%%%%%
\section{quasi-coherent generators}
\label{gensh:SEC}
%%%%%%%%%%%%%%%%%%%%%%%%%%%%%%%%%%%%%%%%%%%%%%%%%%%%%%%%%%%%%%%%%%%%%%%%%%%%%%%%%%

In this preliminary section, we define generating families of finitely presented quasi-coherent sheaves, extend the definition to the relative case, and show the usual permanence properties.
In forthcoming sections we restrict entirely to the case of locally free sheaves but for the sake of completeness we treat the general case here.

\begin{defi}\label{genshglo:DEF:generatingfamily}
A family of quasi-coherent $\shO_X$-modules $(\shG_i)_{i \in I }$ is a \newdef{generating family for $X$} by abuse of notation if it is a family of finitely presented generators in the category of all quasi-coherent $\shO_X$-modules $\Qco(X)$. That is, for every quasi-coherent $\shO_X$-module $\shM$ there exists a surjection $\bigoplus_{i \in I} \shG_i^{\oplus n_i} \twoheadrightarrow \shM$.
\end{defi}

\begin{rem}\label{REM:completenessproperty}
The existence of such a family is equivalent to the so called \emph{completeness property}, saying that every quasi-coherent sheaf is the direct limit of finitely presented ones, or in other words that $\Qco(X)$ is compactly generated. This is known to hold for a vast class of algebraic stacks, including (pseudo-) noetherian and qcqf stacks \cite[4.1]{Ryd15:noetherianapprox}. 
\end{rem}

For schemes, a generating family can be given by a suitable family of ideal sheaves, as seen in the following example.

\begin{ex}[{\cite[Prop. 2.2]{SV04:MR2041778}}]\label{EX:ideal_sheaves}
Suppose that $X$ is a noetherian scheme. 
Let $\shI_1, \dots, \shI_n \subset \shO_X$ be a family of ideal sheaves such that $( X-V(\shI_i) )_{1 \leq i \leq n}$ is an affine open covering.
Then the family of all powers $( \shI_i^{j})_{ j \in \bbN, \,1 \leq i \leq n }$ is generating for $X$. If we put $\shG \coloneq \bigoplus_{i=1}^n \shI_i$, then $(\Sym^k(\shG))_{k \in \bbN}$ is also a generating family.
\end{ex}

The definition of a generating family extends to the relative case analogously to relatively ample invertible sheaves. In order to make this precise we provide a formulation in terms of an adjoint pair of functors.

\begin{defi}
Let $f \colon X \to Y$ be a quasi-compact and quasi-separated morphism of algebraic stacks and let $\shG_I = (\shG_i)_{i \in I}$ be a family of finitely presented $\shO_X$-modules. We define an adjoint pair of functors $(\pb{f_{\shG_I}}, \pf{f_{\shG_I}})$ by
\begin{equation}\label{sumtenspb_functor}
\pb{f_{\shG_I}} \colon \Qco(Y)^I \to \Qco(X), \quad (\shN_i)_{i \in I} \mapsto \bigoplus_{i \in I} \shG_i \otimes_{\shO_X} \pb f \shN_i.
\end{equation}
\begin{equation}\label{pfhomfamily_functor}
\pf{f_{\shG_I}} \colon \Qco(X) \to \Qco(Y)^I, \quad \shM \mapsto (\pf f \shHom_{\shO_X}(\shG_i, \shM))_{i \in I}.
\end{equation}
Note that $\shHom_{\shO_X}(\shG_i, \shM)$ is quasi-coherent because $\shG_i$ is of finite presentation. Using the adjunctions $(\pb f, \pf f)$ and $(\shG_i \otimes \cdot , \shHom_{\shO_X}(\shG_i, \cdot))$, $i \in I$, it is straightforward to check that $(\pb{f_{\shG_I}}, \pf{f_{\shG_I}})$ is indeed an adjoint pair.
\end{defi}

\begin{rem}
For an algebraic stack $X$ that possesses a coarse moduli space $X_0$, the case of singleton families that are generating with respect to the natural morphism $\pi \colon X \to X_0$ was studied in \cite[Section 5]{OS03:MR2007396}.
\end{rem}

We present three equivalent ways of constructing relative resolutions.

\begin{lem}\label{LEM:equivdefs} 
Let $f \colon X \to Y$ be a quasi-compact and quasi-separated morphism of algebraic stacks and let $\shG_I = (\shG_i)_{i \in I}$ be a family of finitely presented quasi-coherent $\shO_{X}$-modules. Then the following properties are equivalent:
\begin{enumerate}
  \item \label{LEM:equivdefs:IT:surjection} Every quasi-coherent $\shO_X$-module $\shM$ admits a surjection 
    \begin{equation}
  \bigoplus_{i \in I} \shG_i  \otimes \pb f \shN_i \twoheadrightarrow \shM,
    \end{equation}
  for some family of quasi-coherent $\shO_Y$-modules $(\shN_i)_{i \in I}$. 
  \item \label{LEM:equivdefs:IT:surj_counit}
  The counit $\varepsilon \colon \pb {f_{\shG_I}} \pf {f_{\shG_I}} \Rightarrow \id_{\Qco(X)}$ evaluated at any quasi-coherent $\shO_X$-module $\shM$ is a surjective map
    \begin{equation}
      \varepsilon(\shM) \colon \bigoplus_{i \in I} \shG_i \otimes_{\shO_{X}} \pb f \pf f \shHom_{\shO_{X}}(\shG_i, \shM) \twoheadrightarrow  \shM.    
    \end{equation}
  \item \label{LEM:equivdefs:IT:pf_faithful}  The functor $\pf {f_{\shG_I}}$ is faithful: for every non-zero morphism $\shM_1 \to \shM_2$ in $\Qco(X)$ there exists  $i \in I$ such that the map
  $$\pf f \shHom(\shG_i, \shM_1) \to \pf f \shHom(\shG_i, \shM_2)$$ 
  is non-zero.

\end{enumerate}
\end{lem}
\begin{proof}
Clearly, \ref{LEM:equivdefs:IT:surj_counit} implies \ref{LEM:equivdefs:IT:surjection}, and the converse holds because,  by adjunction, every map $\varphi \colon \pb{f_{\shG_I}}((\shN_i)_{i \in I}) \to \shM$ factors over the counit $\varepsilon$ by $\pb{f_{\shG_I}}( \varphi^\flat)$, where $\varphi^{\flat} : (\shN_i)_{i \in I} \to \pf {f_{\shG_I}}(\shM)$ is the adjoint of $\varphi$.
The equivalence \ref{LEM:equivdefs:IT:surj_counit} $\Leftrightarrow$ \ref{LEM:equivdefs:IT:pf_faithful} is a formal consequence of adjunction (see \cite[2.3]{Par70:MR0265428}).
\end{proof}

\begin{defi}\label{DEF:rel_gen_family}
Let $f \colon X \to Y$ be a morphism of algebraic stacks. 
A family of quasi-coherent $\shO_X$-modules $\shG_I = (\shG_i)_{i \in I}$  is \newdef{$f$-generating} if $f$ is quasi-compact and quasi-separated, each $\shG_i$ is finitely presented, and the equivalent conditions in Lemma \ref{LEM:equivdefs} hold.
If $Y$ is affine, the definition is the same as in the absolute case (see \ref{genshglo:DEF:generatingfamily}).
We call the family $\shG_I$ \newdef{universally $f$-generating} if for every morphism of algebraic stacks $Y' \to Y$ the family of restricted sheaves 
$\shG_I|_{X \times_Y Y'} \coloneq (\shG_i|_{X \times_Y Y'})_{i\in I}$ is generating for the base change $f_{Y'} \colon X' \to Y'$.
\end{defi}

We begin with the usual sorites for (universally) generating families with respect to morphisms.
\begin{prop}\label{PROP:permanence}
 Let $S$ be an algebraic stack, let $f \colon X \to Y$  be a morphism of algebraic $S$-stacks, and let $\shG_I=(\shG_i)_{i \in I}$ be a family of quasi-coherent $\shO_X$-modules.
\begin{enumerate}
\item \label{PROP:permanence:IT:2iso}
The family $\shG_I$ is (universally) $f$-generating if and only if $\shG_I$ is (universally) $f'$-generating for every $2$-isomorphic morphism $f' \colon X \to Y$. 
\item \label{PROP:permanence:IT:quaff} 
The singelton family $\shO_X$ is universally $f$-generating if $f$ is quasi-affine (for instance, if $f$ is an affine, finite, or quasi-finite representable separated morphism, a finite-type monomorphism, a quasi-compact open immersion, or a closed immersion).
\item \label{PROP:permanence:IT:local}
\emph{$\fpqc$-local on the target:} Let $(S_\alpha \to S)$ be an $\fpqc$ covering family (or a faithfully flat family, $f$ is quasi-compact and quasi-separated, and each $\shG_i$ is finitely presented). If the restricted family  $\shG_I|_{X_{(S_\alpha)}}$ is generating for $f_{(S_\alpha)} \colon X_{(S_\alpha)} \to Y_{(S_\alpha)}$ and each $\alpha \in A$, then $\shG_I$ is $f$-generating.
\item \label{PROP:permanence:IT:basechange}
\emph{Base change:} Let $S' \to S$ be a morphism of algebraic stacks such that $S$ has quasi-affine diagonal. 
If $\shG_I$ is $f$-generating, then the restricted family $\shG_I|_{X_{(S')}}$ is generating for $f_{(S')}\colon X_{(S')} \to Y_{(S')}$. 
\item  \label{PROP:permanence:IT:composition}
\emph{Composition:}
Let $g \colon Y \to Z$ be a morphism of algebraic $S$-stacks, and let $\shE_J = (\shE_j)_{j\in J}$ be a family of quasi-coherent $\shO_Y$-modules.
If $\shG_I$ is (universally) $f$-generating and if $\shE_J$ is (universally) $g$-generating, then the family
 \[\shG_I \otimes \pb f {\shE_J} \coloneq \left(\shG_i \otimes \pb f \shE_j \right)_{(i,j )\in I \times J} \]
is (universally) generating for $g \circ f$.
\item  \label{PROP:permanence:IT:lcp}
\emph{Left-cancellation property:}
Suppose that $g$ is quasi-separated (resp. $\Delta_g$ quasi-affine).
If $\shG_I$ is $g \circ f$-generating (resp.~univ. $g \circ f$-generating), then $\shG_I$ is $f$-generating (resp.~universally $f$-generating).
\item \label{PROP:permanence:IT:products}
\emph{Products:}
Let $f_\alpha \colon X_\alpha \to Y_\alpha$, $\alpha = 1,2$, be morphisms of algebraic $S$-stacks and denote by $p_\alpha \colon X_1 \times_S X_2 \to X_\alpha$ the projections. If $\shG_{I_\alpha}^{(\alpha)}$, $\alpha = 1,2$, are universally $f_\alpha$-generating families on $X_\alpha$, then the family \[\shG_{I_1}^{(1)} \boxtimes_S \shG_{I_2}^{(2)} \coloneq \left(\pb \pr_1 \shG_{i_1}^{(1)} \otimes \pb \pr_2 \shG_{i_2}^{(2)}\right)_{(i_1, i_2) \in  I_1 \times I_2}\]
 is universally generating for $f_1 \times_S f_2 \colon X_1 \times_S X_2 \to Y_1 \times_S Y_2$.
\item \label{PROP:permanence:IT:reduction}
\emph{Reduction:}
 If $\shG_I$ is (universally) $f$-generating, then the restricted family $\shG_I|_{X_{\red}}$ is (universally) generating  for $f_{\red} \colon X_{\red} \to Y_{\red}$.
\end{enumerate}
\end{prop}

\begin{rem}
Let $P$ be a property of finitely presented sheaves which is local and satisfies $\fpqc$-descent (e.g. ``locally free''). Then the permanence properties shown in Proposition \ref{PROP:permanence}, carry over to (universally) relatively generating families of finitely presented sheaves satisfying $P$ \emph{mutatis mutandis}.
\end{rem}
\begin{proof}[Proof of \ref{PROP:permanence}]---

\emph{Proof of \ref{PROP:permanence:IT:2iso}}: The universal case reduces to the non-universal case, which follows from Lemma \ref{LEM:equivdefs}.\ref{LEM:equivdefs:IT:pf_faithful} because faithfulness of a functor is preserved and reflected under $2$-isomorphisms.

\bigskip

\emph{Proof of \ref{PROP:permanence:IT:local}}: 
We may assume that $S=Y$ by restricting the faithfully flat covering $(S_\alpha \to S)$ along $Y \to S$. Given a faithfully flat covering $u_\alpha \colon  Y_\alpha \to Y$, consider for each $\alpha$ the induced $2$-cartesian square
\begin{equation}
\xymatrix{
  X_\alpha \ar[r]^{v_\alpha} \ar[d]_{f_\alpha}   \ar@{}[dr]|{\square} &	X \ar[d]^f	\\
  Y_\alpha  \ar[r]^{u_\alpha}  & Y
}
\end{equation}
Let $i \in I$ be given. By assumption, $f_\alpha$ is quasi-compact and quasi-separated, and each $\pb v_\alpha \shG_i$ is finitely presented.
Then $f$ is quasi-compact and quasi-separated and $\shG_i$ is finitely presented by $\fpqc$ descent (or by assumption). 
Thus the following diagram consists of well-defined functors:
\begin{equation}
\xymatrix{
\Qco(X) \ar[r]^{\pb{v_\alpha}} \ar[d]_{\shHom_{\shO_X}(\shG_i, \cdot )} & \Qco(X_\alpha) \ar[d]^{\shHom_{\shO_{X_\alpha}}(\pb {v_\alpha}\shG_i, \cdot )} \\
\Qco(X) \ar[r]^{\pb{v_\alpha}} \ar[d]^{\pf f} & \Qco(X_\alpha) \ar[d]^{\pf{ f_\alpha}} \\
\Qco(Y) \ar[r]^{\pb{u_\alpha}} & \Qco(Y_\alpha) \\
}
\end{equation}
The upper square is $2$-commutative since $\shG_i$ and $\pb {v_\alpha} \shG_i$ are of finite presentation and $\pb {v_\alpha}$ commutes with the internal hom's by flatness. 
The lower square is $2$-commutative by flat base change \cite[13.1.9]{LM00:MR1771927}.
Thus, the whole diagram is $2$-commutative. 
The assertion follows now by a simple diagram chase:
Since $(v_\alpha)$ is a faithfully flat covering family for $X$, the induced pullback functor  $v_A \colon \Qco(X) \to \prod_\alpha \Qco(X_\alpha), \shM \mapsto \pb {v_\alpha}(\shM)$ is  faithful. 
Similarly we get a faithful functor $u_A \colon \Qco(Y) \to \prod_\alpha \Qco(Y_\alpha)$.
For each $\alpha$ let $\pb{v_\alpha}\shG_I$ be the family of restricted sheaves $(\pb{v_\alpha}\shG_i \mid i \in I)$, which is generating for $f_\alpha$ by hypothesis. Thus, $\pf {((f_\alpha)_{\pb{v_\alpha}\shG_I})} \colon \Qco(X_\alpha) \to \Qco(Y_\alpha)^I$ is faithful for each $\alpha$. 
We conclude that the composition
\[ \Qco(X) \xrightarrow{v_A} \prod_\alpha \Qco(X_\alpha) \xrightarrow{\left(\pf {((f_\alpha)_{\pb{v_\alpha}\shG_I})}\right)_{\alpha}} \prod_\alpha \Qco(Y_\alpha)^I \xrightarrow{\simeq} \left(\prod_\alpha \Qco(Y_\alpha)\right)^I \]
is faithful, \emph{a fortiori} this holds for the $2$-isomorphic functor
\[ \Qco(X) \xrightarrow{\pf {f_{\shG_I}}} \Qco(Y)^I \xrightarrow{(\pb{u_A})^I} \left(\prod_\alpha \Qco(Y_\alpha)\right)^I \]
By the left cancellation property for faithful functors, we conclude that $\pf {f_{\shG_I}}$ is faithful, too. Thus $\shG_I$ is $f$-generating.

\bigskip

 \emph{Proof of \ref{PROP:permanence:IT:quaff}:} As the property ``quasi-affine'' is stable under arbitrary base change, it suffices to show that $\shO_X$ is $f$-generating. By hypothesis there is a smooth covering $Y' \to Y$ by a scheme $Y'$ such that the base change $f' \colon X'\coloneq X \times_Y Y' \to Y'$ is quasi-affine. Thus, $\shO_{X'}$ is $f'$-ample \cite[5.1.2]{egaII}, hence $f'$-generating. So by \ref{PROP:permanence:IT:local} we conclude that $\shO_X$ is $f$-generating.

\bigskip

\emph{Proof of \ref{PROP:permanence:IT:composition}}:
It suffices to treat the non-universal case by replacing $g \circ f \colon X \to Y \to Z$ for a given $Z' \to Z$ with the base change $g' \circ f' \colon X' \to Y' \to Z'$ and using the isomorphisms
$(\shG_i \otimes_{\shO_X} \pb f \shE_j)|_{X'} \simeq \shG_i|_{X'} \otimes_{\shO_{X'}} \pb{f'}\shE_j|_{Y'}$ for all $(i,j) \in I \times J$.
By assumption $f$ and $g$ are quasi-compact and quasi-separated, so the same holds for $g \circ f$. 
Let $(i,j) \in I \times J$ be given.
Since $\shG_i$ and $\shE_j$ are of finite presentation, so is $\shG_i \otimes \pb f \shE_j$.
Then we get a diagram of well-defined functors:
\begin{equation}\label{composition_eq1}
\xymatrix@C=5pc@R=1pc{ 
  \Qco(X) \ar[r]^{\shHom(\shG_i, \cdot)}\ar[dr]_{\shHom(\shG_i \otimes \pb f \shE_j, \cdot) \qquad \qquad } & \Qco(X)  \ar[d]^{\shHom(\pb f \shE_j, \cdot)}\ar[r]^{\pf f} & 
  \Qco(Y) \ar[d]^{\shHom(\shE_j, \cdot)} \\
  &  
  \Qco(X) \ar[r]^{\pf f} \ar[dr]_{\pf {(g \circ f)}} & 
  \Qco(Y) \ar[d]^{\pf g} \\
  & 
  & 
  \Qco(Z)  
}
\end{equation}
The upper left triangle is $2$-commutative by adjunction of  $\shG_i \otimes\cdot $ and $\shHom_{\shO_X}(\shG_i, \cdot)$ in $\Qco(X)$. The lower triangle is $2$-commutative by definition.
The square is $2$-commutative since it corresponds by adjunction to 
the isomorphism $\pb f ( \shE_j \otimes_{\shO_Y} \cdot) \simeq \pb f \shE_j \otimes_{\shO_X} \pb f(\cdot)$. 
Thus, the whole diagram is $2$-commutative.
It follows that the composition 
\[
 \left(\pf {g_{\shE_J}} \right)^I \circ \pf {f_{\shG_I}}  \colon \Qco(X) \xrightarrow{\pf {f_{\shG_I}}} \Qco(Y)^I \xrightarrow{\left(\pf {g_{\shE_J}} \right)^I} (\Qco(Z)^J)^I \xrightarrow{\simeq} \Qco(Z)^{I \times J}
\]
is $2$-isomorphic to the functor 
\[
 \pf {\left((g\circ f)_{\shG_I \otimes \pb f \shE_J}\right)}  \colon \Qco(X) \to \Qco(Z)^{I \times J}.
\]
By hypothesis, $\pf {f_{\shG_I}}$ and $\pf {g_{\shE_J}}$ are faithful. Then the constant functor $\left(\pf {g_{\shE_J}} \right)^I$ is faithful and so is the composition $\left(\pf {g_{\shE_J}} \right)^I \circ \pf {f_{\shG_I}} \simeq \pf {\left((g\circ f)_{\shG_I \otimes \pb f \shE_J}\right)}$ as required.
\bigskip

\emph{Proof of \ref{PROP:permanence:IT:lcp}}: Let us first prove the non-universal case.
By assumption, $\shG_I$ is a family of finitely presented $\shO_X$-modules. Since $g \circ f$ is quasi-compact and quasi-separated, it follows that $f$ is quasi-compact and quasi-separated since $\Delta_g$ is by assumption.
Consider now diagram \eqref{composition_eq1} with the singleton family $\shE = \shO_Y$. Since we do not assume that $g$ is quasi-compact, the functor $\pf g \colon \Qco(Y) \to \Qco(Z)$ does not make sense, but $\pf g \colon \Qco(Y) \to \Mod(Z)$ does. Thus, we replace $\Qco(Z)$ with $\Mod(Z)$, so that the lower triangle is again well-defined. As above, a diagram chase shows us that 
\[
 (\pf g)^I \circ  \pf {f_{\shG_I}} \colon \Qco(X) \to \Qco(Y)^I \to \Mod(Z)^{I}
\]
is $2$-isomorphic to 
\[
 i^I \circ \pf {(g \circ f)_{\shG_I}} \colon \Qco(X) \to \Qco(Z)^I \to \Mod(Z)^I
\]
By hypothesis $\pf {(g \circ f)_{\shG_I}}$ is faithful. Also, the constant functor $i^I$ is faithful.
Thus, $(\pf g)^I \circ  \pf {f_{\shG_I}}$, a \emph{fortiori} the factor $\pf {f_{\shG_I}}$ is faithful as asserted.

For the universal case we use the standard argument that $f$ factors up to $2$-isomorphism as the composition of the upper horizontal morphisms of the following two $2$-cartesian squares:
\begin{equation}
\xymatrix{
X \ar[r]^{\Gamma_f}\ar[d]^f \ar @{} [dr] |{\square}	& X \times_Z Y \ar[d]^{f\times1}
	&& X \times_Z Y \ar[r]^{q}\ar[d]^{p} \ar @{} [dr] |{\square}		& Y \ar[d]^g \\
Y \ar[r]^{\Delta_g}		& Y \times_Z Y
	&& X \ar[r]^{g \circ f}						& Z \\
}
\end{equation}
By hypothesis $\Delta_g$ and hence $\Gamma_f$ are quasi-affine. Thus, $\shO_X$ is universally $\Gamma_f$-generating by  \ref{PROP:permanence:IT:quaff}.
Since $\shG_I$ is universally $g \circ f$-generating by assumption, $\pb p \shG_I$ is universally $q$-generating. So by  applying \ref{PROP:permanence:IT:composition} to $f = q \circ \Gamma_f$ we conclude that $\lbrace\shO_X\rbrace \otimes \pb{\Gamma_f} (\pb p \shG_I)$ is universally $f$-generating. But in light of the identity $p \circ \Gamma_f = \id_X$ this means that $\shG_I$ is universally $f$-generating.
 
\bigskip

\emph{Proof of \ref{PROP:permanence:IT:basechange}}: 
Choose a smooth covering family $S'_\alpha \to S'$ of affine scheme $S'_\alpha$. Then each composition $S'_\alpha \to S' \to S$ is a quasi-affine morphism because $\Delta_{S/\bbZ}$ is quasi-affine.
Then the base change $Y_{(S'_\alpha )} \to Y$ and $X_{(S'_\alpha )}  \to X$ are quasi-affine, too, so that $\shO_{Y_{(S'_\alpha )}}$ and $\shO_{X_{(S'_\alpha )}}$ are relatively generating by \ref{PROP:permanence:IT:quaff}.
It follows that the family of restricted sheaves $\shG_I|_{X_{(S'_\alpha)}}$ is generating for the composition $X_{(S'_\alpha)} \to X \to Y$ by \ref{PROP:permanence:IT:composition}, \emph{a fortiori} for the $2$-isomorphic morphism $X_{(S'_\alpha)} \to Y_{(S'_\alpha)} \to Y$ and hence for $X_{(S'_\alpha)} \to Y_{(S'_\alpha)}$ by \ref{PROP:permanence:IT:lcp} because $Y_{(S'_\alpha)} \to Y$ has quasi-affine diagonal.
Then \ref{PROP:permanence:IT:local} implies that $\shG_I|_{(S')}$ is $f_{(S')}$-generating.

\bigskip

\emph{Proof of \ref{PROP:permanence:IT:products}}: 
The product morphism $f_1 \times_S f_2$ is the decomposition of the upper horizontal morphisms of the following $2$-cartesian squares, where $p_\alpha, q_\alpha, r_\alpha$ denote the projections on the $\alpha$-th factor:
\begin{equation}
\xymatrix{
X_1 \times_S X_2 \ar[r]^{(f_1,\id_{X_2})}\ar[d]^{p_1} \ar @{} [dr] |{\square}	& Y_1 \times_S X_2 \ar[d]^{q_1}
	&& Y_1 \times_S X_2 \ar[r]^{(\id_{Y_1},f_2)}\ar[d]^{q_2} \ar @{} [dr] |{\square}		& Y_1 \times_S Y_2 \ar[d]^{r_2} \\
X_1 \ar[r]^{f_1}		& Y_1
	&& X_2 \ar[r]^{f_2}						& Y_2 \\
}
\end{equation}
Then the family $\pb{p_1}\shG_{I_1}^{(1)}$ is universally $(f_1,\id_{X_2})$-generating, and the family $\pb {q_2}\shG_{I_2}^{(2)}$ is  universally $(\id_{Y_1},f_2)$-generating. 
Hence $\pb{p_1}\shG_{I_1}^{(1)} \otimes \pb{(f_1,\id_{X_2})}\left(\pb {q_2} \shG_{I_2}^{(2)}\right)$ is universally $f_1 \times_S f_2$-generating by \ref{PROP:permanence:IT:composition}. But due to $q_2 \circ (f_1,\id_{X_2}) \simeq p_2$ we can identify the right factor of the latter tensor product with the family $\pb{p_2}\shG_{I_2}^{(2)}$. This proves the assertion.

\bigskip

\emph{Proof of \ref{PROP:permanence:IT:reduction}}:
The morphisms $f$ and $f_{\red}$ fit in a $2$-commutative square, where the horizontal morphisms are closed immersions:
\begin{equation}
\xymatrix{
  X_\red \ar[r]^{v} \ar[d]_{f_\red} &	X \ar[d]^f	\\
  Y_\red  \ar[r]^{u}  & Y
}
\end{equation}
Since $\shO_{X_\red}$ is universally $v$-generating and $u$ has quasi-affine diagonal, the assertion is a consequence of \ref{PROP:permanence:IT:composition} applied to $f \circ v$ and \ref{PROP:permanence:IT:lcp} applied to $u \circ f_{\red}$.
\end{proof}

\begin{cor}\label{COR:generating=univgenerating}
Let $f \colon X \to Y$ be a quasi-compact and quasi-separated morphism of algebraic stacks.
If $Y$ has quasi-affine diagonal, then every $f$-generating family is universally $f$-generating.
\end{cor}

\begin{rem}\label{REM:permanence} For families of quasi-coherent sheaves on algebraic stacks without quasi-affine diagonal the properties ``universally generating'' and ``generating'' do not coincide. For a quasi-separated morphism $f \colon X \to Y$, the structure sheaf $\shO_X$ is generating for $\Delta_f \colon X \to X \times_Y X$, by applying Proposition \ref{PROP:permanence}.\ref{PROP:permanence:IT:lcp}  to the factorization $\id_X = \pr_1 \circ \Delta_f$. However, $\shO_X$ is not necessarily universally $\Delta_f$-generating.
To give a counterexample, let $A \overset{\pi}\rightarrow \Spec k$ be an abelian scheme of positive dimension.
Then the trivial torsor $p \colon \Spec k \to \clsp A$ induces a $2$-cartesian square
\begin{equation}
\xymatrix{
A \ar[r]\ar[d]^\pi \ar@{}[dr]|{\square} & \clsp A \ar[d]^{\Delta} \\
\Spec k \ar[r]^-{(p,p)}	& \clsp A \times_k \clsp A 	 \\
}
\end{equation}
Although $\shO_{\clsp A}$ is $\Delta$-generating, $\shO_{A}$ is not $\pi$-generating (equivalently $\pi$-ample) since $A$ is not quasi-affine. Hence, $\shO_{\clsp A}$ is not universally $\Delta$-generating.
\end{rem}

As expected, the property ``universally generating'' can be tested over affines:
\begin{prop}\label{PROP:test_on_affines}
Let $f \colon X \to Y$ be a morphism of algebraic stacks and let $\shG_I = (\shG_i)_{i \in I}$ be a family of quasi-coherent $\shO_X$-modules. Then the following properties are equivalent:
\begin{enumerate}
  \item \label{PROP:test_on_affines:IT:1} 
    $\shG_I$ is universally $f$-generating.
  \item \label{PROP:test_on_affines:IT:2} 
    For every morphism $\Spec A \to Y$, the family of restricted sheaves $\shG_I|_{X_A}$ is generating for $X_A$.
  \item \label{PROP:test_on_affines:IT:3} 
    There exists an $\fpqc$ covering family $(Y_\alpha \to Y)$  of algebraic stacks $Y_\alpha$ with quasi-affine diagonal such that each restricted family  $\shG_I|_{(Y_\alpha)}$ is generating for $X_{(Y_\alpha)} \to Y_\alpha$.
\end{enumerate}
\end{prop}
\begin{proof} The implications
\ref{PROP:test_on_affines:IT:1} $\Rightarrow$ \ref{PROP:test_on_affines:IT:2} $\Rightarrow$ \ref{PROP:test_on_affines:IT:3} are trivial.
For \ref{PROP:test_on_affines:IT:3} $\Rightarrow$ \ref{PROP:test_on_affines:IT:1} note that for each $\alpha$, the restriction $\shG_I|_{(Y_\alpha)}$ is universally generating for $X_{(Y_\alpha)} \to  Y_\alpha$ by Corollary \ref{COR:generating=univgenerating} using that $Y_\alpha$ has quasi-affine diagonal. Therefore $\shG_I$ is  universally $f$-generating by $\fpqc$  descent (Proposition \ref{PROP:permanence}.\ref{PROP:permanence:IT:local}).
\end{proof}

The following establishes descent of the completeness property along finite, flat, finitely presented surjections. It seems to be known before only for \'etale maps \cite{Tho87:MR893468}.

\begin{prop}\label{Genqcofp_Rel_FinFlat:PROP:descent}
Let $f \colon X \to Y$ be a finite, faithfully flat and finitely presented morphism, and let $g \colon Y \to Z$ be a quasi-compact and quasi-separated morphism of algebraic stacks.
If $\shG_I$ is a (universally) $g \circ f$-generating family of $\shO_X$-modules, then the family of $\shO_Y$-modules $\pf f \shG_I = (\pf f \shG_i)_{i \in I}$ is (universally) $g$-generating. 
\end{prop}
\begin{proof}
It suffices to treat the non-universal case by applying an appropriate base change.
So let us assume that $\shG_I$ is a $g \circ f$-generating family of $\shO_X$-modules.
Now, we invoke Grothendieck duality for finite morphisms. Recall that $\pf f$ preserves finitely presented sheaves because $f$ is finite and locally free, and that $\pf f$ has a right adjoint $\us f$ defined by $\pf f \us f ( \cdot) = \shHom_{\shO_Y}(\pf f \shO_X, \cdot)$. Then the adjunction formula
$\pf f \shHom_{\shO_X}( \cdot, \us f (\cdot)) = \shHom_{\shO_Y}(\pf f (\cdot), \cdot)$ implies that for each $i \in I$ holds 
\(
 \pf g \circ \shHom_{\shO_Y}(\pf f \shG_i, \cdot ) \simeq \pf g \circ \pf f \circ \shHom_{\shO_X}(\shG_i, \cdot) \circ \us f
\)
as isomorphism of functors $\Qco(Y) \to \Qco(Z)$.
Using Proposition \ref{LEM:equivdefs}.\ref{LEM:equivdefs:IT:pf_faithful} one can see that $\pf f$ maps $g \circ f$-generating families to $g$-generating families if $\us f$ is faithful.
The latter is equivalent to the property that the counit $\pf f \us f (\shM) \to \shM$ is surjective for every quasi-coherent $\shO_Y$-module $\shM$. By applying $\shHom_{\shO_Y}( \cdot, \shM)$ to the canonical map $\varphi_f \colon \shO_Y \to \pf f \shO_X$, we see that this happens precisely if $\varphi_f$ is an $\fppf$ locally split monomorphism of quasi-coherent $\shO_Y$-modules. The latter is true by faithfully flatness of $f$ because $\varphi_f$ is a map of $\shO_Y$-algebras.
\end{proof}

We do not know a general descent method for non-finite flat affine coverings. The main obstacle is that the pushforward of a finitely presented quasi-coherent sheaf is no longer finitely presented. The following technical lemma is a reminiscence of this approach and will be helpful to construct generating families on low dimensional stacks.

\begin{lem}\label{LEM:non-finite_pushdown}
Let $f \colon Y \to X$ be an affine and faithfully flat morphism of algebraic stacks such that $Y$ is quasi-affine.
Then every quasi-coherent $\shO_X$-module $\shM$ is a quotient of a quasi-coherent $\shO_X$-submodule $\shN \subset \pf f \shO_Y^{(I)}$ for some set $I$.
\end{lem}
\begin{proof}
First, we may identify $\shM$ with a subsheaf of $\pf f \pb f \shM$ via the unit $\delta \colon \shM \to \pf f \pb f \shM$, which is injective since $f$ is faithfully flat. 
As $Y$ is quasi-affine, we can pick a surjection 
$\shO_Y^{(I)} \twoheadrightarrow \pb f\shM$.
Then the pushforward $\psi \colon \pf f \shO_Y^{(I)} \twoheadrightarrow \pf f \pb f\shM$ is surjective since $f$ is affine. It follows that the preimage $\shN \coloneq \psi^{-1}(\shM)$ has the desired properties. \end{proof}

\begin{cor}
Let $X$ be a reduced quasi-compact algebraic stack with affine diagonal. Then every quasi-coherent sheaf is a quotient of a torsionfree quasi-coherent sheaf.
\end{cor}

%%%%%%%%%%%%%%%%%%%%%%%%%%%%%%%%%%%%%%%%%%%%%%%%%%%%%%%%%%%%%%%%%%%%%%%%%%%%%%%%%
\section{Pinching schemes }
\label{SEC:pinching}
%%%%%%%%%%%%%%%%%%%%%%%%%%%%%%%%%%%%%%%%%%%%%%%%%%%%%%%%%%%%%%%%%%%%%%%%%%%%%%%%%

Recall that every quasi-compact algebraic space $X$ is finitely parametrized by a \emph{scheme}, saying that it admits a finite and finitely presented surjection $f \colon Z \to X$ from a scheme $Z$ (see \cite[16.6]{LM00:MR1771927} for the noetherian case and \cite[Thm. B]{Ryd15:noetherianapprox} for the general case).  
In this section we show that if every fiber of $f$ is contained in an affine open subset, then $X$ is representable by a scheme.

\begin{defi}\label{DEF:AFscheme}
An algebraic space $X$ is an \newdef{$\AF$-scheme} (or satisfies the \newdef{Chevalley--Kleiman property}) if the following condition is satisfied:
\begin{enumerate}[label=(AF), ref=(C\arabic{*})]
\item Every finite set of points $x_1, \dots, x_n \in \pts{X}$ is contained in a Zariski open neighborhood that is representable by an affine scheme.
\end{enumerate}
\end{defi}

\begin{rem}\label{REM:AF} ---
\begin{enumerate}
\item \label{REM:AF:IT:sep}Every $\AF$-scheme is separated.
\item \label{REM:AF:IT:finite+AF=affine}Every $\AF$-scheme with finitely many points is affine.
\item \label{REM:AF:IT:normal+AF=quasi-projective} Every normal $\AF$-scheme of finite type over an algebraically closed field is quasi-projective \cite[Cor. 2]{Ben13}.
\item \label{REM:AF:IT:ascent}If $f \colon Y \to X$ is a strongly representable morphism of algebraic spaces such that $Y$ admits a relatively ample invertible sheaf and $X$ is an $\AF$-scheme, then $Y$ is an $\AF$-scheme. In particular, this holds if $f$ is affine or quasi-affine.
\end{enumerate}
\end{rem}

The main result of this section is to prove that the $\AF$ property descends along morphisms that are separated, surjective, universally closed and have finite topological fibers. It is a stronger variant of Theorem \ref{THM:AF_descends_simple}.

\begin{thm}\label{THM:AF_descends}
Let $f \colon Y \to X$ be a morphism of algebraic spaces that is separated, surjective, universally closed and has finite topological fibers, i.e.~the topological space $|f^{-1}(x)|$ is finite for every $x \in X$ (discrete or not). Then $X$ is an $\AF$-scheme if and only if $Y$ is an $\AF$-scheme.
\end{thm}

As D. Rydh pointed out, $f$ is integral by the following Lemma. In particular, $f$ has discrete fibers.

\begin{lem}\label{LEM:univclosed_sep_finittopfibers_then_integral}
Let  $f \colon Y \to X$ be a universally closed and separated morphism of algebraic spaces. If the topological fibers of $f$ are finite, then $f$ is integral.
\end{lem}
\begin{proof}
Note that $f$ is representable and quasi-compact because the fibers are quasi-compact and $f$ is closed.
By applying \cite[Thm. 8.3]{Ryd15:noetherianapprox} the morphism $f$ factors  as an integral and surjective morphism $Y \to Y_0$ followed by a proper morphism $Y_0 \to X$. Since $Y \to X$ has finite fibers, so has $Y_0 \to X$. It follows that $Y_0 \to X$ is finite (Zariski's Main Theorem). Hence $Y \to X$ is integral.
\end{proof}

A key step of the proof of Theorem \ref{THM:AF_descends} is the verification of the following global representability criterion.

\begin{prop}\label{PROP:finitely_param_algspace}
Let $Z \to X$ be an integral surjective morphism of algebraic spaces. If $Z$ is a quasi-compact and quasi-separated scheme that admits an ample invertible sheaf, then $X$ is representable by a quasi-compact and separated $\AF$-scheme. If $X$ is noetherian and normal, then $X$ admits an ample invertible sheaf.
\end{prop}

\begin{rem}
 The result is well known if $X$ is a noetherian normal scheme  using the norm map \cite[6.6.2]{egaII}, or if $f$ is flat, finite and finitely presented, or if $f$ is a quotient map $Z \to Z/G$ of a geometric quotient by a finite group. If $Z$ is affine, then $X$ is affine by Chevalley's Theoreom for affines (\cite[8.1]{Ryd15:noetherianapprox}, or \cite{Knu71:MR0302647} in the noetherian case).
\end{rem}

\begin{proof}[Proof of Proposition \ref{PROP:finitely_param_algspace}]
Let us call $X$ \emph{finitely parametrized} if there exists such a finite surjective morphism $p \colon Z \to X$ from a scheme with an ample line bundle. We frequently use that this property ascends along finite maps $X' \to X$.

Note that $X$ is quasi-compact since $Z$ is quasi-compact and $p$ is surjective. $X$ is separated because $Z$ is separated and $p \colon Z \to X$ universally closed.

\emph{Step 1. Reduction to the case that $p$ is finite and finitely presented:}
$Z$ is the filtered projective limit $\varprojlim_\lambda Z_\lambda$ of integral and finitely presented (hence finite) algebraic $X$-spaces $Z_\lambda$  with affine bonding maps $Z_\lambda \to Z_\mu$ because $X$ is pseudo-noetherian \cite[Thm. A]{Ryd15:noetherianapprox}. Then for sufficiently large $\lambda$, each $Z_\lambda$ is a quasi-compact and separated scheme \cite[Thm. C.(iii)]{Ryd15:noetherianapprox}. By taking even larger $\lambda$, every ample sheaf $\shL$ on $Z$ descends to a compatible family of ample sheaves $\shL_{\lambda}$ on $Z_{\lambda}$. This is stated in \cite[C.8]{TT90:MR1106918} in case that the base $X$ is an affine scheme, but the proof also applies in the general case since the property ``affine'' can be approximated over an arbitrary quasi-compact algebraic stack \cite[Thm. C.(i)]{Ryd15:noetherianapprox}.
So we may assume that $f$ is finite and finitely presented by replacing $p$ with $Z_\lambda \to X$.

\emph{Step 2. Noetherian and normal case:}  $X$ is a geometric quotient of a noetherian normal scheme $X'$ by a finite group $G$ \cite[16.6.2]{LM00:MR1771927}. It follows that $X'$ is finitely parametrized: Since $X' \to X$ is finite, the pullback $\pr_1 \colon Z' \coloneq Z \times_X X' \to Z$ is finite, so that $Z'$ admits an ample line bundle $\shL'$ by hypothesis on $Z$. Since $\pr_2 \colon Z' \to X'$ is finite and surjective, $\shL \coloneq \norm_{Z'/X'}(\shL')$ is an ample $\shO_{X'}$-module \cite[6.6.2]{egaII}, showing that $X'$ is an $\AF$-scheme \cite[4.5.4]{egaII}. Then it is well-known that the geometric quotient $X = X'/G$ is representable by an $\AF$-scheme and that  $\norm_{X'/X}(\shL)$ is an ample $\shO_X$-module.

\emph{Step 3. Final step:}
By approximating $X$ and $p$, we may assume that $X$ is of finite type over $\bbZ$ (the reduction step in the proof of Chevalley's Theorem \cite[8.1]{Ryd15:noetherianapprox} applies literally). In particular, $X$ is noetherian and Nagata.
If $X_\red$ is an $\AF$-scheme, then $X$ is an $\AF$-scheme as a consequence of Chevalley's Theorem. Therefore we may assume that $X$ is reduced since $X_\red$ is finitely parametrized.
The normalization $f \colon X' \to X$ is finite since $X$ is Nagata, hence $X'$ is normal, noetherian and finitely parametrized. 
By step 2 we know that $X'$ is representable by an $\AF$-scheme.
Let  $i \colon Y \subset X$ be the closed subspace defined by the ideal $\Ann \coker (\shO_X \to \pf f \shO_{X'}) $, set $i' \colon Y' = Y \times_X X' \subset  X'$ and $g \coloneq f|_{Y'} \colon Y' \to Y$. Then $X$ is the pushout of $i'$ and $g$ in the category of algebraic spaces because $f$ has schematically dense image (see Lemma \ref{ResProp_GenOX_Chevalley:LEM:algspaces_pinching} below).
Since $Y \subset X$ is a proper subspace that is finitely parametrized, by noetherian induction we may assume that $Y$, and hence $Y'$, is an $\AF$-scheme. Thus, the pushout $X_0 \coloneq X' \sqcup_{Y'} Y$ exists already in the category of ringed spaces and is an $\AF$-scheme since $X'$ and $Y$ are $\AF$-schemes \cite[5.4]{Fer03:MR2044495}. 

We claim that $X=X_0$ and since $X$ is the pushout in the category of algebraic spaces it is enough to prove that $X$ is a scheme. Let $f_0 \colon X' \to X_0$ be the quotient map, which is finite, surjective and schematically dominant. Then by the universal property it factors as $f_0=h \circ f$ for some map of algebraic spaces $h \colon X \to X_0$. In order to see that $X$ is a scheme, we may assume that $X_0$ is affine by taking a Zariski covering of affine open subschemes of $X_0$. Then $X'$ is affine using that $f_0$ is affine. Consequently, $X$ is affine by Chevalley's Theorem since $f$ is finite and surjective, proving the assertion.
\end{proof}

The following preparatory lemma is folklore but stated for lack of reference.
\begin{lem}\label{ResProp_GenOX_Chevalley:LEM:algspaces_pinching}
Given a morphism of algebraic spaces $f \colon X' \to X$, the closed immersion $i \colon Y \hookrightarrow X$  defined by the conductor ideal $\Ann_{\shO_X} \coker (\shO_X \to \pf f \shO_{X'})$ and the preimage $Y' \coloneq f^{-1}(Y)$  with closed immersion $i' \colon Y' \hookrightarrow X'$ and restriction $g=f|_{Y'} \colon Y' \to Y$ gives rise to a cartesian square:
 \begin{equation}\label{ResProp_GenOX_Chevalley:LEM:algspaces_pinching:EQ:1}
\xymatrix{
Y' \ar[r]^{i'} \ar[d]^g  & X' \ar[d]^{f} \\
Y \ar[r]^{i}  & X }
\end{equation}
If $f$ is finite and schematically dominant (i.e.~$\shO_X \to \pf f \shO_{X'}$ is injective), then the square is cocartesian in the category of algebraic spaces.
\end{lem}
\begin{proof} 
In case that $X$ and hence $X'$, $Y$, $Y'$ are affine schemes, the conductor square is cocartesian in the category of algebraic spaces  (\cite[\S 1.1]{Fer03:MR2044495} and \cite[Proof of Thm. 2.2.2]{CLO12}).
It follows that for every \'etale covering $u \colon U \to X$ with $U$ affine, one recovers $U$ as the pushout of $g_U \colon Y'_U \to Y_U$ and $i'_U \colon Y'_U \to X'_U$.

In order to see that the square is cocartesian,
let $h \colon X' \to T$, $j \colon Y \to T$ be given morphisms satisfying $hi'=jg$. We have to construct a unique map $t \colon X \to T$ with $tf=h$ and $ti=j$. 
Suppose there are two maps $t_1, t_2 \colon X \to T$ satisfying this condition. Then $t_1 u = t_2 u$ by uniqueness of the former case, hence $t_1=t_2$ since $\Hom(\cdot, T)$ is a separated presheaf. This shows uniqueness. Regarding the existence, observe that $X'_U \to T$ and $Y'_U \to T$ factor over a unique map $t' \colon U \to T$. It gives rise to two morphisms $t' \circ \pr_{\alpha} \colon U \times_X U \to U \to T$, $\alpha=1,2$, and both satisfy the compatibility condition after restricting \eqref{ResProp_GenOX_Chevalley:LEM:algspaces_pinching:EQ:1} along the \'etale covering $U \times_X U \to X$. So by uniqueness, we infer $t' \circ \pr_1 = t' \circ \pr_2$. Since $\Hom(\cdot, T)$ is a sheaf, there is a map $t \colon X \to T$ with $t u = t'$. The condition $tf=h$ (resp.~$ti=j$) is local over $X'$ (resp.~$Y$), hence follows by restricting \eqref{ResProp_GenOX_Chevalley:LEM:algspaces_pinching:EQ:1} along $u$.
\end{proof}

\begin{cor}\label{COR:intcovering_schematic_points}
Let $f \colon Y \to X$ be an integral surjective morphism of algebraic spaces. A finite set of points $P \subset \pts X$ is contained in an affine open subspace if and only if $f^{-1}(P) \subset |Y|$ is contained in an affine open subspace. 
\end{cor}
\begin{proof}
The condition is clearly necessary since $f$ is affine. Conversely, suppose that $V \subset Y$ is an affine Zariski open neighborhood of $f^{-1}(P)$. 
Since $f$ is closed, the set  $U = X - f(Y-V)$ is an open neighborhood of $P$ such that  $f^{-1}(P) \subset f^{-1}(U)\subset V$. Using that $P$ is quasi-compact we may replace $U$ with a quasi-compact open subset. Then $f^{-1}(U)$ is quasi-compact too and hence a quasi-affine open subscheme of $V$. Then $U$ is representable by an $\AF$-scheme by Proposition \ref{PROP:finitely_param_algspace}. Since $P \subset U$ and $P$ is finite we conclude that there exists an affine open subset $P \subset W \subset U$. 
\end{proof}

\begin{proof}[Proof of Theorem \ref{THM:AF_descends} and Theorem \ref{THM:AF_descends_simple}]
By Lemma \ref{LEM:univclosed_sep_finittopfibers_then_integral} the morphism $f$ is integral, hence affine. Thus, if $X$ is an $\AF$-scheme then so is $Y$.
Conversely, if $Y$ is an $\AF$-scheme then by Corollary \ref{COR:intcovering_schematic_points} $X$ must be an $\AF$-scheme, too.  
\end{proof}

%%%%%%%%%%%%%%%%%%%%%%%%%%%%%%%%%%%%%%%%%%%%%%%%%%%%%%%%%%%%%%%%%%%%%%%%%%%%%%%%
\section{Global generation of sheaves and quasi-affiness}
\label{SEC:quasi-affiness}
%%%%%%%%%%%%%%%%%%%%%%%%%%%%%%%%%%%%%%%%%%%%%%%%%%%%%%%%%%%%%%%%%%%%%%%%%%%%%%%%

In this section we show that for a quasi-compact and quasi-separated algebraic stack $X$ with affine stabilizer groups, the condition that every quasi-coherent sheaf is globally generated implies that $X$ is a quasi-affine scheme. This is well-known if $X$ is a quasi-compact and quasi-separated scheme (\cite[1.7.16]{egaIV_1}).

\begin{prop}\label{PROP:quaff=genOX+affinestabilizers}
A quasi-compact and quasi-separated morphism of algebraic stacks $f \colon X \to Y$ is quasi-affine if and only if the following conditions are satisfied:
\begin{enumerate}
 \item \label{PROP:quaff=genOX+affinestabilizers:IT:1} $\shO_X$ is universally $f$-generating.
 \item \label{PROP:quaff=genOX+affinestabilizers:IT:2} $f$ has affine relative stabilizer groups at geometric points, i.e.~the geometric fibers of the relative inertia \(I_f \to X\) are affine (equiv.~quasi-affine).
This holds for instance, if $f$ has quasi-affine diagonal (e.g. if $\Delta_f$ is quasi-finite and separated).
\end{enumerate}
\end{prop}
\begin{proof}
The conditions are necessary by Proposition \ref{PROP:permanence}.\ref{PROP:permanence:IT:quaff} so let us verify the sufficiency. 
Both assumptions \ref{PROP:quaff=genOX+affinestabilizers:IT:1} and \ref{PROP:quaff=genOX+affinestabilizers:IT:2} are stable under base change, and the assertion is local over $Y$. Therefore, we may assume that $Y=\Spec(A)$ is affine and that $X$ is quasi-compact and quasi-separated, by replacing $Y$ with an appropriate smooth covering. 

First, we show that $X$ is representable.
For that, it suffices to show that $f$ has representable geometric fibers.
Therefore, we may assume that $Y$ is the spectrum of an algebraically closed field, by applying base change with a given point. Now, we have to show that for every point $x \colon \Spec k \to X$, the stabilizer group $G_x$ is trivial, which is an affine (algebraic) group scheme by assumption on $I_f$. Let $\xi \in |X|$ be the point induced by $x$. If $x' \colon \Spec k' \to X$
denotes another representative, then $G_{x'}$ and $G_x$ are isomorphic over some common field extension. Thus, if $G_{x'}$ is trivial, then so is $G_x$ by $\fpqc$-descent, showing that it suffices to find some representative of $\xi$ with trivial stabilizer group.

Using that $\xi \in |X|$ is algebraic \cite[Theorem B.2]{Ryd11:devissage}, there exists a representative  $x \colon \Spec k \to X$ that factors over the residual gerbe $\shG_\xi$ by an epimorphism $\overline x \colon \Spec k \twoheadrightarrow \shG_\xi$ followed by a quasi-affine monomorphism $\shG_\xi \hookrightarrow X$. The gerbe $\shG_\xi$ is an algebraic stack of finite type over the residue field $k(\xi)$, which is the sheafification of $\shG_\xi$.
It follows that there exists a finite field extension $k(\xi) \subset L$ such that $\shG_\xi \otimes_{k(\xi)} L \simeq \clsp G_{x'}$, where $G_{x'} \to \Spec L$ is the stabilizer group at the induced representative $x' \colon \Spec L \to \shG_\xi \hookrightarrow X$ of $\xi$. 
The upshot is that the composition $\clsp G_{x'} \to \shG_\xi \hookrightarrow X$ is a quasi-affine map, so that $\shO_{\clsp G_{x'}}$ is relatively generating.
Since $\shO_X$ is generating for $X$,  we conclude that $\shO_{B G_{x'}}$ is an absolute generator for $\clsp G_{x'}$ by Proposition \ref{PROP:permanence}\ref{PROP:permanence:IT:composition}. 
But then $G_{x'} \to \Spec L$ is the trivial algebraic $L$-group scheme because every linear representation is generated by the trivial representation. Therefore, $X$ is representable by an algebraic space. 

In order to see that $X$ is a quasi-affine scheme, take a finite, finitely presented and surjective morphism $p \colon Z \to X$ for some \emph{scheme} $Z$ \cite[Thm B]{Ryd15:noetherianapprox}. 
Since $p$ is quasi-affine, $\shO_Z = \pb p \shO_X$ is generating for $Z$, so that $Z$ is quasi-affine \cite[1.7.16]{egaIV_1}. But then
 $X$ must be a scheme by Theorem \ref{THM:AF_descends}. By the former argument, we conclude that $X$ is quasi-affine.
\end{proof}

\begin{rem}
In case that $f$ is representable and $Y$ has quasi-affine diagonal,  Proposition \ref{PROP:quaff=genOX+affinestabilizers} was proven in \cite[6.2]{AE12} by different methods.
If $X$ is noetherian and normal it can be deduced from the proof of Totaro's Theorem \cite[1.1]{Tot04:MR2108211}.
\end{rem}

\begin{cor}
A morphism of algebraic stacks $f \colon X \to Y$ has quasi-affine diagonal $\Delta_f$ if and only if $\shO_X$ is universally $\Delta_f$-generating.
\end{cor}

\begin{cor}
A morphism of algebraic stacks $f \colon  X \to Y$  is quasi-affine if and only if $\shO_X$ is universally generating for $f $ and $\Delta_f$.
\end{cor}

%%%%%%%%%%%%%%%%%%%%%%%%%%%%%%%%%%%%%%%%%%%%%%%%%%%%%%%%%%%%%%%%%%%%%%%%%%%%%%%%
\section{The resolution property}
\label{SEC:resprop}
%%%%%%%%%%%%%%%%%%%%%%%%%%%%%%%%%%%%%%%%%%%%%%%%%%%%%%%%%%%%%%%%%%%%%%%%%%%%%%%%

In this section we define the resolution property of a morphism in terms of locally free generating sheaves and recall the example classes where it is known to hold.
From now on we implicitly assume that every vector bundle has constant rank.

\begin{defi}\label{resproprel:DEF:resproprel}
An algebraic stack $X$ has the \emph{resolution property} if $X$ is quasi-compact and quasi-separated and if there exists a generating family of locally free $\shO_X$-modules.
We say that a morphism $ f\colon X \to Y$ of algebraic stacks has the \newdef{resolution property}, or that \newdef{$X$ has the resolution property over $Y$ (relative to $f$)}, if  $f$ is quasi-compact and quasi-separated and if there exists a universally $f$-generating family of locally free $\shO_X$-modules (see Definition \ref{DEF:rel_gen_family}).
\end{defi}

\begin{rem}\label{genshglo:REM:Totaros_resprop}
For a noetherian algebraic stack this definition is equivalent to Totaro's \cite{Tot04:MR2108211}, saying that $X$ has the resolution property if and only if every coherent sheaf is a quotient of a coherent locally free sheaf, by taking the family of all vector bundles (up to isomorphism) because $X$ has the completeness property (cf. Remark \ref{REM:completenessproperty}).
\end{rem}

Let us give the usual sorites for this class of morphisms.

\begin{prop}\label{resproprel:PROP:permanence} ---
\begin{enumerate}
  \item \label{resproprel:PROP:permanence_examples}  
Every affine, finite, or quasi-finite representable separated morphism, finite-type monomorphism, quasi-compact immersion, or more generally quasi-affine morphism has the resolution property.
  \item \label{resproprel:PROP:permanence_basechange} Let $Y' \to Y$ be a morphism.
    If a morphism $f \colon X \to Y$ has the resolution property, then so has the base change $f'\colon X' \to Y'$.
  \item \label{resproprel:PROP:local}
    Let $f \colon X \to Y$ be morphism and let $Y' \to Y$ be an $\fpqc$ morphism. If the base change $f' \colon X' \to Y'$ has the resolution property given by a family of locally free $\shO_{X'}$-modules $\shG_I' = (\shG_i')_{i \in I}$ endowed with a descent datum relative to $X' \to X$ (i.e.~isomorphisms $\sigma_{i} \colon \pb {\pr_1 }\shG_i' \xrightarrow{\simeq} \pb {\pr_2}\shG_i'$ for each $i \in I$, where $\pr_\alpha \colon X' \times_X X' \to X$, that satisfy the cocycle condition over $X' \times_X X' \times_X X'$), then $f$ has the resolution property and there is a universally $f $-generating family $\shG_I = (\shG_i)_{ i \in I}$ such that $\shG_i|_{X'} \simeq \shG_i'$ for each $i \in I$.
  \item \label{resproprel:PROP:permanence_products}
    If two morphisms $f \colon X_\alpha \to Y_\alpha$, $\alpha=1,2$, over an algebraic stack $S$, have the resolution property, then so has $f \times_S g \colon X_1 \times_S X_2 \to Y_1 \times_S Y_2$.  
\item \label{resproprel:PROP:permanence_composition}
    If $f \colon X \to Y$ and $g \colon Y \to Z$ have the resolution property, then so has $g \circ f$.
  \item \label{resproprel:PROP:permanence_leftcancellation}
    Suppose that $\Delta_g$ is quasi-affine. If $g \circ f$  has the resolution property, then so has $f$.
  \item \label{resproprel:PROP:permanence_finflatdescent} 
  Suppose that $f$ is finite, faithfully flat and finitely presented.
  If  $g \circ f$ has the resolution property, then so has $g$.
  \item \label{resproprel:PROP:permanence_reduction} If $f \colon X \to Y$ has the resolution property, then so has $f_\red \colon X_\red \to Y_\red$.
\end{enumerate}
\end{prop}
\begin{proof}
 The property ``locally free and finitely presented'' of quasi-coherent sheaves is stable under taking pullbacks or tensor products and satisfies descent with respect to $\fpqc$ covers. Thus Proposition \ref{PROP:permanence} holds \emph{mutatis mutandis} for generating and universally generating families of locally free finitely presented quasi-coherent sheaves. 
 From this one easily deduces properties \ref{resproprel:PROP:permanence_examples}-\ref{resproprel:PROP:permanence_leftcancellation} and \ref{resproprel:PROP:permanence_reduction}.
Finally, property \ref{resproprel:PROP:permanence_finflatdescent} is a consequence of Proposition \ref{Genqcofp_Rel_FinFlat:PROP:descent}.
\end{proof}

\begin{lem}[Finite fppf groupoids]\label{resprop:LEM:finitefppf_groupoids}
Let $R \rightrightarrows U$ be a finite, faithfully flat, finitely presented groupoid of algebraic $S$-spaces. If $U$ (and hence $R$) satisfies the resolution property over $S$, then so does the quotient stack $X=[R \rightrightarrows U]$. 
\end{lem}
\begin{proof}
The quotient map $q \colon U \to X$ is finite, finitely presented and faithfully flat. Thus, Proposition \ref{resproprel:PROP:permanence}.\ref{resproprel:PROP:permanence_finflatdescent} applies.
\end{proof}

\begin{cor}\label{resprop:EX:finite-groups}
Let $G \to S$ be a flat, finite and finitely presented (equiv.~finite, locally free) group algebraic space  over an algebraic space $S$ that satisfies the resolution property. Then the classifying stack $\clsp G$ has the resolution property.
\end{cor}
\begin{proof}
The trivial $G$-torsor $S \to \clsp G$ is finite, finitely presented and faithfully flat.
\end{proof}

\begin{rem}
This result is well-known if $G \to S$ is \'{e}tale \cite[2.14]{Tho87:MR893468}. 
\end{rem}

\begin{cor}\label{resprop:COR:CM_covers}
Let $X$ be a regular algebraic stack that admits a finite, finitely presented surjection $f \colon Y \to X$ such that $Y$ is Cohen-Macaulay and satisfies the resolution property. Then $X$ has the resolution property.
\end{cor}
\begin{proof}
The regularity properties of $Y$ and $X$ imply that $f$ is flat.
\end{proof}

\begin{lem}[Stacks with regular noetherian covers of dimension $\leq 1$]\label{COR:stacks_with_dim_0or1_have_resprop}
Let $f \colon Y \to X$ be an affine faithfully flat morphism of algebraic stacks. Suppose that $Y$ is a noetherian regular scheme of dimension $\leq 1$. Then $X$ has the resolution property.
\end{lem}
\begin{proof}
Since $Y$ is quasi-compact and has affine diagonal (using that $\dim Y \leq 1$), we may replace $Y$ with an affine $Zariski$ covering and hence assume that $Y$ is affine. Then by Lemma \ref{LEM:non-finite_pushdown} and \cite[15.4]{LM00:MR1771927}, it suffices to resolve coherent subsheaves $\shM \subset \pf f \shO_Y^{\oplus n}$, $n \in \bbN$. But these are already locally free by flat descent because $\pb f \shM \subset \pb f \pf f \shO_Y^{\oplus n}$ is a finitely generated subsheaf of a torsion-sheaf free, and hence locally free since $\dim Y \leq 1$ and $Y$ is regular.
\end{proof}

\begin{ex}[Schemes]
Given a noetherian scheme $X$, the resolution property is known to hold in the following cases:
\begin{enumerate}
\item $X$ is divisorial. That is, every point $x \in X$ admits an affine open neighborhood that is the non-vanishing locus of a global section $s \in \Gamma(X, \shL)$ for some invertible sheaf $\shL$ (\cite{Bor63:MR0153683}, \cite{Bor67:MR0219545}). This is true, if $X$ is quasi-affine, or quasi-projective over a noetherian ring \cite[5.3.2]{egaII} (including all algebraic curves and all separated algebraic surfaces with finitely many isolated singularities that are contained in an affine open \cite[Cor. 4, p.328]{Kle66:MR0206009}). This also holds if $X$ is normal and $\bbQ$-factorial with affine diagonal (\cite[1.3]{BS03:MR1970862}, and the case of separated, regular noetherian schemes is due to Kleiman and independently Illusie \cite[II.2.2.7]{sga6})).
\item $X$ is separated and of finite type over a Dedekind ring and $\dim(X) \leq 2$ (\cite[5.2]{Gro12:respropsurfaces}, and for normal separated algebraic surfaces \cite[2.1]{SV04:MR2041778}). In dimension $\geq 2$ there exist normal, separated algebraic schemes that have no non-trivial invertible sheaves, and hence are not divisorial (see \cite{Schro99:MR1726231} for algebraic surfaces).
\end{enumerate}
\end{ex}

\begin{ex}[Classifying stacks of algebraic group schemes]\label{genshglo:EX:Thomasons_groupstacks_results} Given an affine, flat and finitely presented group scheme $\pi \colon G \to S$ over a noetherian and separated scheme, Thomason \cite{Tho87:MR893468} verified the absolute resolution property for $\clsp G$ in the following cases:
\begin{enumerate}
\item \label{genshglo:EX:Thomasons_groupstacks_results:IT:1} $S$ is regular with $\dim(S)\leq 1$.
\item \label{genshglo:EX:Thomasons_groupstacks_results:IT:2} $S$ is a regular with $\dim(S)=2$ and $\pf \pi \shO_G$ is a locally projective $\shO_S$-module; for instance, if $G \to S$ is smooth with connected fibers.
\item \label{genshglo:EX:Thomasons_groupstacks_results:IT:3} $S$ satisfies the resolution property, $G \to S$ is reductive and either 
$G$ is semisimple, or $S$ is normal, or the radical and coradical of $G$ are isotrivial (i.e.~diagonalizable on a finite \'etale cover of $S$).
\end{enumerate}
\end{ex}

\begin{rem}
By Totaro's Theorem and its generalization to arbitrary algebraic stacks (Theorem  \ref{THM:resprop=basic} below) we know that a quasi-compact and quasi-separated algebraic stack with affine stabilizer groups that satisfies the resolution property must have affine diagonal.
So every algebraic stack with quasi-affine and non-affine diagonal does not have the resolution property. 
As an example, glue two copies of $\bbA^2_k$ at the complement of the origin to get a scheme with quasi-affine and non-affine diagonal. Similarly, take the quasi-affine group scheme $G$ obtained from $\bbZ/2\bbZ \to \bbA_k^2$ by removing the origin in the non-identity component, then the classifying stack $\clsp G$ has quasi-affine but not affine diagonal.

There is an example of an algebraic stack with affine diagonal that does not have the resolution property. It is the $\Gm$-gerbe over a complex algebraic surface $Y$  
corresponding to a non-torsion element of the cohomological Brauer group $\Coho^2_{\et}(Y, \Gm)$.

We do not know if every algebraic stack with quasi-finite and affine diagonal has the resolution property, even in case of normal, separated algebraic schemes of an algebraically closed field (like toric threefolds, see \cite{Pay09:MR2448277}). 

\'{E}tale locally, every algebraic stack with quasi-finite and locally separated diagonal has the resolution property \cite[Cor. 2.7]{Ryd15:noetherianapprox}.
\end{rem}

\begin{rem}
If an algebraic stack $Y$ is fibered over an algebraic stack $X$ by means of a morphism $f \colon Y \to X$, then the question whether the resolution property holds or not, can be broken down to the relative resolution property of $f$ and the resolution property of the base $X$.
For example, from this point of view one can tackle the \emph{equivariant resolution property} of an algebraic space $Y$, acted on by an affine, flat and finitely presented group scheme $G$. It says that every quasi-coherent $\shO_Y$-$G$-comodule is a quotient of a direct sum of locally free and finitely presented $\shO_Y$-$G$-comodules. Now, quasi-coherent $\shO_Y$-$G$-comodules correspond to quasi-coherent sheaves on the quotient stack $X \coloneq [Y/G]$.
 On the one hand, the affine and faithfully flat quotient map $Y \to X$ is a $G$-torsor, and we get a $G$-fibration of $Y$ over the base $X$. On the other hand, the classifying morphism $X \to \clsp G$ imposes on $X$ a $Y$-fibration over $\clsp G$.
\end{rem}

\begin{prop}\label{PROP:resprop_equivar}
Let $S$ be an algebraic space and let $G \to S$ be an affine, flat and finitely presented algebraic group space that acts on an algebraic $S$-space $Y$.  Then the following conditions are equivalent:
\begin{enumerate}
 
 \item \label{PROP:resprop_equivar2} The classifying map $X=[Y/G] \to \clsp G$ has the resolution property,
 \item \label{PROP:resprop_equivar3} $Y$ has a family of $G$-linearized locally free $\shO_Y$-modules of finite type that is universally generating for $Y$ over $S$.
\end{enumerate}

Moreover, if $\clsp G \to S$ has the resolution property, then the conditions are equivalent to:
\begin{enumerate}
\item \label{PROP:resprop_equivar1} $[Y/G]$ has the resolution property over $S$,
\end{enumerate}
\end{prop}
\begin{proof}
By definition of the classifying stack $\clsp G$ there exists a $2$-cartesian square of $S$-stacks, where the vertical arrows are fppf-coverings:
\begin{equation}\label{clsp_fiber_square}
 \xymatrix{ Y \ar[r]\ar[d] &S \ar[d] \\
            [Y/G] \ar[r]  & \clsp G}
\end{equation}
Hence, a generating family of quasi-coherent sheaves for $[Y/G] \to \clsp G$ restricts to a quasi-coherent family on $Y \to S$ with descent datum, which is equivalent to giving a $G$-linearization. Conversely, every $G$-linearized family of quasi-coherent sheaves for $Y \to S$ descends to a family of quasi-coherent sheaves on $[Y/G]$. During this restriction and descent process the property of being a relative generating family of finitely presented locally free sheaves is preserved.
\end{proof}

\begin{rem}
An algebraic stack $X$ is a \emph{global quotient stack} if it is equivalent to an algebraic stack $[Y/\GL_n]$, where $Y$ is an algebraic space acted on by $\GL_n$ for some integer $n\geq 0$. It seems to be unknown whether the resolution property descends along the quotient map $q \colon Y \to [Y/\GL_n]$, in general. Though, if $Y$ is normal scheme that  admits an ample family of invertible sheaves $\lbrace \shL_i \rbrace_{i=1}^n$, then the powers $\lbrace \shL_i^{\otimes m} \rbrace$ admit a $\GL_n$-linearization for $m$ sufficiently large by Sumihiros Theorem \cite[Thm. 1.6]{Sum75:MR0387294}, and hence descend to invertible sheaves, whose duals form a generating family for $[Y/G] \to \clsp G$. Consequently, a large class of global quotient stacks $[Y/G]$ satisfy the resolution property.
\end{rem}

The following Lemma is a courtesy of D. Rydh and follows essentially \cite[4.1-4.2]{Tot04:MR2108211}.
\begin{lem}\label{LEM:resprop_closed_affines}
Let $X$ be a quasi-compact and quasi-separated algebraic stack. Assume that $X$ has the resolution property and affine stabilizers at closed points.
Then there is a vector bundle $\shE$ on $X$ such that the total space $Z$ of the corresponding $\GL_n$-torsor has affine stabilizers. In particular, $X$ has affine stabilizers at \emph{all} points.
\end{lem}
\begin{proof}. A vector bundle $\shE$ gives a \emph{twisted} representation of the inertia stack $I_X$ of $X$ (a morphism $I_X$ to a twisted form of $\GL_n$ over $X$) with kernel a (closed) subgroup $H \subseteq I_X$ such that the pull-back of $H$ to $Z$ equals the inertia stack of $Z$. For example, $H=0$ if $Z$ is an algebraic space.

Pick a closed point $x$. Pick a vector bundle $\shF$ over the residual gerbe $\shG_x$ such that the stabilizer action is faithful (possible since the stabilizer is affine). Consider the inclusion morphism $i \colon \shG_x \to X$ and note that $\pf i \shF$ is quasi-coherent.
Since $X$ has the resolution property, there is a vector bundle $\shE$ and a morphism $\shE \to \pf i {\shF}$ such that $\pb i{\shE} \to \pb i {\pf i {\shF}} \to \shF$ is surjective. This means that the representation of $\shG_x$ on $\shE$ is faithful, that is $H_x \coloneq \pb i H = 0$. By upper semi-continuity, $H$ is quasi-finite in an open neighborhood of $x$.

Repeating this for all closed points $x$ gives a finite number of vector bundles such that $\shE_1 \oplus ...\oplus \shE_n$ has a corresponding $\GL_n$-torsor $Z$ with quasi-finite inertia. In particular, $Z$ has affine stabilizers.

Thus, the stabilizer groups of $X$ are extensions of subgroups of $\GL_n$ with stabilizer groups of $Z$, hence affine.
\end{proof}

%%%%%%%%%%%%%%%%%%%%%%%%%%%%%%%%%%%%%%%%%%%%%%%%%%%%%%%%%%%%%%%%%%%%%%%%%%%%%%%%
\section{Tensor generators and Totaro's Theorem}
\label{SEC:tensorgen}
%%%%%%%%%%%%%%%%%%%%%%%%%%%%%%%%%%%%%%%%%%%%%%%%%%%%%%%%%%%%%%%%%%%%%%%%%%%%%%%%

In this section, we define the property of a vector bundle to be a tensor generator, and show that it is equivalent to the property that its associated frame bundle has quasi-affine total space, when the stabilizer groups are affine. Moreover, we prove the generalization of Totaro's Theorem \cite{Tot04:MR2108211} to the general relative case and give the proof of Theorem \ref{THM:thmA} at the end of the section.

\begin{defi}
Let $f \colon X \to Y$ be a morphism of algebraic stacks, and let $\shV$ on $X$ be a vector bundle of constant rank $n$ with associated frame bundle $p \colon \shIsom_X(\shO_X^n ,\shV) = F(\shV) \to X$. We shall say that $\shV$ is a \newdef{tensor generator for $X$ over $Y$} (or just \emph{$f$-tensor generating}), if the family of all $p$-locally split subsheaves $\shG \subset P(\shV, \shV^{\vee})$ is a universally $f$-generating family of quasi-coherent sheaves, where $P$ runs over all polynomials $P \in \bbN[t,s]$. Here, by $p$-locally split, we mean that the inclusion $\shG \subset P(\shV, \shV^{\vee})$ admits a left-inverse, when restricting to the total space $F(\shV)$.
Moreover, we say that $\shV$ is a \newdef{strong tensor generator for $X$ over $Y$} if the family $P(\shV, \shV^{\vee})$ of all polynomial expressions is a universally generating family.
\end{defi}

\begin{rem}
Suppose that $Y=\Spec k$ for a field $k$, and that $X=\clsp G$ where $G$ denotes an algebraic group scheme. Then the definition above coincides with the one in Tannaka theory \cite[6.16]{Del90:MR1106898} when we identify quasi-coherent $\shO_{\clsp G}$-modules with $k$-linear $G$-representations. 
\end{rem}

\begin{rem}\label{REM:tensor_generator_permanence}
The property ``tensor generator'' (resp. ``strong tensor generator'') is stable under base change. Also, the property ``strong tensor generator'' is local on the base. Moreover, it is stable on the source under pullback with quasi-affine morphisms $X' \to X$. 

As a consequence of  the following Theorem \ref{THM:tensor_generator}, the property ``tensor generator'' is local on the base if the morphism has relatively affine stabilizer groups because the property ``quasi-affine'' of morphisms is local on the base.
\end{rem}

\begin{thm}\label{THM:tensor_generator}
 Let $f \colon X \to Y$ be a quasi-compact morphism of algebraic stacks. Given a vector bundle $\shV$ with associated frame bundle $p \colon F(\shV) \to X$, the following properties are equivalent:
\begin{enumerate}

 \item \label{THM:tensor_generator1} 
\begin{enumerate}
\item $\shV$ is a tensor generator for $f \colon X\to Y$.
\item The relative inertia stack $I_f \to X$ has affine fibers.
\end{enumerate} 
  \item \label{THM:tensor_generator2} $f\circ p \colon F(\shV) \to Y$ is quasi-affine, or equivalently, the classifying morphism $c_{\shV} \colon X \to \clsp \GL_{n,Y}$ is quasi-affine.
\end{enumerate}
\end{thm}
\begin{proof}\ref{THM:tensor_generator1} $\Rightarrow$ \ref{THM:tensor_generator2}:
Let $(\shG_i \subset P_i(\shV, \shV^\vee))_{i \in I}$ be a universally $f$-generating family of $p$-locally split subsheaves. Then $( \pb p \shG_i)_{i \in I}$ is universally $f \circ p$-generating since $p$ is affine.
But each $\pb p \shG_i$ is a direct summand of the free sheaf $\pb p P(\shV, \shV^\vee)$, showing that $\shO_{F(\shV)}$ is universally $f \circ p$-generating. Moreover, since $I_f \to X$ has affine fibers and $p$ is affine, we observe that $I_{f \circ p} \to X$ has affine fibers. Thus, $f \circ p$ is quasi-affine by Proposition \ref{PROP:quaff=genOX+affinestabilizers}.

\ref{THM:tensor_generator2} $\Rightarrow$ \ref{THM:tensor_generator1}:
First, note that $I_f \to Y$ has affine fibers because the stabilizer groups of $X \to Y$ are subgroups of $\GL_n$ on each fiber.

Let $\shE$ be the vector bundle on $\clsp \GL_{n}$ that corresponds to the standard representation. Then $\pb{ (q \circ c_{\shV} )} \shE \simeq \shV$, where $q \colon  \clsp \GL_{n,Y} \to \clsp \GL_{n}$ denotes the projection that is given by the base change $Y \to \Spec(\bbZ)$. So, if $\shE$ is a tensor generator for $\clsp \GL_{n}$, then $\shV$ is one for $f$ by Remark \ref{REM:tensor_generator_permanence}.

This reduces to the case that $X=\clsp \GL_{n}$ over $Y = \Spec \bbZ$, where $\shV$ is the standard representation. It follows from the proof of \cite[Theorem 3.5]{Wat79:MR547117} that $\pf p \shO_Y$ is a union of vector bundles $\shV_{\alpha} \subset \pf p \shO_Y$, such that each $\shV_{\alpha}$ is a quotient of some polynomial expression $P_{\alpha}(\shV, \shV^\vee)$. Namely, by choosing a basis for the free $\bbZ$-module $\pb p \shV$, then $\pf p \shO_Y$ can be identified with the Hopf algebra $\bbZ[x_{11}, \dots, x_{nn}, \det^{-1}]$ and $\shV_{\alpha}$ with the comodule $\det^{-s_\alpha}\lbrace q \in \bbZ[x_{11}, \dots, x_{nn}] \mid \deg q \leq r_\alpha \rbrace$, $s_\alpha, r_\alpha \geq 0$.

In order to prove that $\shV$ is a tensor generator, pick a quasi-coherent sheaf $\shF$ on $\clsp \GL_{n}$. We may assume that $\shF$ is locally free of constant rank $m$ since $\clsp \GL_{n}$ has the resolution property. 
Then holds $\pb p \shF \simeq \shO_Y^{\oplus m}$ because every locally free sheaf on $Y=\Spec(\bbZ)$ is free.
So, we can identify $\shF$ with a $p$-locally split subsheaf of $\pf p \shO_Y^{\oplus m}$ via the unit $\shF \to \pf p \pb p \shF$, which is $p$-locally split by adjunction. As $\shF$ is of finite type, the inclusion $\shF \subset \pf p{\shO_Y^{\oplus m}}$ factors as $\shF \subset \shV_\alpha^{\oplus m} \subset \pf p \shO_Y^{\oplus m}$ for every sufficiently large $\alpha$. Also, $\shF \subset \shV_\alpha^{\oplus m}$ is $p$-locally split, so that $\shV_\alpha^{\oplus m}/\shF$ is a vector bundle.

Choose a surjection $P(\shV, \shV^\vee) \twoheadrightarrow \shV_\alpha^{\oplus m}$, and let $\shG$ be the preimage of $\shF$. Then $\shF$ is a quotient of $\shG$. Moreover, the induced map $P(\shV, \shV^\vee)/\shG \to \shV_\alpha^{\oplus m}/\shF$ is an isomorphism, showing that $P(\shV, \shV^\vee)/\shG$ is a vector bundle, so that $\shG \subset P(\shV, \shV^\vee)$ is $p$-locally split. Thus, $\shV$ is a tensor generator.
\end{proof}

If the structure group of the vector bundle $\shV$ is linearly reductive, then a generating family can be deduced from $\shV$ without taking subsheaves. Recall that an group scheme $G$ over $Y$ is linearly reductive if $\clsp G \to Y$ is cohomologically affine \cite[Def. 12.1]{Alp13:goodmoduli}.

\begin{prop}\label{PROP:linred_tengen} 
With the preceding notations, suppose that $\shV$ is a tensor generator and that the associated bundle of $\GL_{n,Y}$-frames $p \colon F(\shV) \to X$ is induced by a $G$-torsor $\pi \colon Z \to X$, for some flat linearly reductive subgroup scheme $G \subset \GL_{n,Y}$. 
Then $\shV$ is a strong tensor generator. 
\end{prop}
\begin{proof}
First, note that the classifying morphism $c_{\pi} \colon X \to \clsp G$ is quasi-affine because the composition $X \xrightarrow{c_{\pi}} \clsp G \to \clsp \GL_n$ is and $\clsp G \to \clsp \GL_n$, being representable, has quasi-affine diagonal. 
The upshot is that $\pb {c_{\pi}}$ preserves generating families, which reduces the statement to the case $X$ is $\clsp G$ for some linearly reductive group scheme $G$ over an affine base $Y$. 

Pick a $p$-locally split subsheaf $\shG \subset P(\shV,\shV^\vee)$ and denote by $\shC$ the quotient. Then $\pb p \shC$ is a direct summand of the free sheaf $\pb p P(\shV,\shV^\vee)$, showing that $\shC$ is locally free. Then $\Ext^1_{\clsp G}(\shC, \shG) = \Ext^1_{\clsp G}(\shO_{\clsp G}, \shC^\vee \otimes \shG) = H^1(\clsp G,\shC^\vee \otimes \shG)$. The latter cohomology group vanishes because $G$ is linearly reductive. It follows that the short exact sequence $0 \to \shG \to P(\shV,\shV^\vee) \to \shC \to 0$ splits.
\end{proof}

It is well-known that $\GL_{n, Y}$ is linearly reductive if $Y$ is of characteristic zero:
\begin{cor}\label{COR:tengen_char_zero}
Let $f \colon X \to Y$ be a morphism with $Y$ of characteristic zero and relatively affine stabilizer groups. Then every tensor generator is a strong tensor generator.
\end{cor}

A split vector bundle  $\shV=\bigoplus_{j=1}^{n} \shL_j$ has linearly reductive structure group $\Gm^n$:
\begin{cor}\label{COR:split_tensor_generator}
Let $(\shL_j)_{j=1}^n$ be a family of invertible sheaves on $X$. Then the following properties are equivalent:
\begin{enumerate}
\item $\bigoplus_{j=1}^n \shL_{j}$ is a tensor generator for $f$.
\item $\bigoplus_{j=1}^n \shL_{j}$ is a strong tensor generator for $f$.
\item The family $(\shL_1^{\otimes \ell_1} \otimes \cdots \otimes \shL_n^{\otimes \ell_n})_{\ell_1, \dots, \ell_n \in \bbZ}$ is universally $f$-generating.
\item The total space $Z$ of the associated $\Gm^n$-torsor $Z \to X$ is quasi-affine over $Y$.
\end{enumerate}
In particular, every invertible tensor generator is a strong tensor generator.
\end{cor}

\begin{rem}
The corollary generalizes \cite[Thm. 1]{Hau02:MR1900763} to algebraic stacks that are not irreducible algebraic varieties.
\end{rem}
\begin{prop}
Let $f \colon X \to Y$ be a morphism of algebraic stacks such that the inertia $I_f \to X$ has affine fibers. Let $\shV$ be a vector bundle on $X$. Then $\shV$ is a tensor generator for $f$ if and only if $\shV|_{X_\red}$ is a tensor generator for $f_{\red} \colon X_\red \to Y_\red$.
\end{prop}
\begin{proof}
The total space of the frame bundle $F(\shV)$ is quasi-affine over $Y$ if and only if $F(\shV|_{X_\red})$ is quasi-affine over $Y_{\red}$ \cite[Cor. 8.2]{Ryd15:noetherianapprox}. Hence, the result is a direct consequence of Theorem \ref{THM:tensor_generator}.
\end{proof}

\medskip

Let us finally generalize Totaro's theorem to the general relative case.

\begin{thm}\label{THM:resprop=basic}
 Let $f \colon X \to Y$ be a morphism of algebraic stacks with $Y$ quasi-compact. Then the following conditions are equivalent:
\begin{enumerate}
\item \label{THM:resprop=basic:IT:resprop} 
\begin{enumerate}
 \item $f$ has the resolution property, and
 \item the relative inertia stack $I_{f} \to X$ has affine fibers (for instance, if $\Delta_f$ is quasi-affine).
\end{enumerate}
\item \label{THM:resprop=basic:IT:basic} For sufficiently large $n \geq 0$ the morphism $f$ admits a factorization
 \begin{equation}
 \xymatrix{
  X \ar[dr]_{f} \ar[r]^-{g} & \clsp \GL_{n,Y} \ar[d]^q \\
  & Y
}
\end{equation}
where $g$ is quasi-affine, which is the classifying morphism of the frame bundle for some vector bundle $\shV$ of rank $n$, and $q$ is the structure morphism. In particular, the diagonal $\Delta_f$ is affine.
\end{enumerate}
\end{thm}

\begin{proof}
Let $f \colon X \to Y$ be a morphism of algebraic stacks. Suppose that $f$ factors by a quasi-affine morphism $ X \to \clsp \GL_{n,Y}$ followed by the projection $\clsp \GL_{n,Y} \to Y$.
Then both morphisms have the resolution property and so has the composition $f$. Moreover, both morphisms have affine diagonal, so $f$ has too, and we conclude that the inertia $I_f = X \times_{X \times_Y X} X \to X$ is affine.

Conversely, suppose that $f$ has the resolution property and that $I_f \to X$ has affine fibers at geometric points.
Let $(\shV_i)_{i \in I}$ be a universally $f$-generating family of vector bundles on $X$, for instance the family of all vector bundles on $X$ up to isomorphism. We have to show that there exists a quasi-affine morphism $X \to \clsp \GL_{n,Y}$. 
For every finite subset $J \subset I$ the $X$-fiber product 
$p_J \colon F_J \coloneq (\prod_{/X})_{i \in J} F(\shV_i) \to X$ is an affine morphism. Let $(F_J \to F_K )_{K \subset J}$ be the natural inverse system for the family $(F_J \to X)$. Then the inverse limit of $X$-stacks $F \coloneq \varprojlim F_J$ is an algebraic stack over $X$ because the bonding maps $F_J \to F_K$ are affine.
The projection $p \colon F\to X$ is an affine morphism and has the property, that $\pb p (\shV_i)$ is trivial for every $i \in I$. So $\shO_{F}$ is universally $f \circ p$-generating. Since $I_f \to X$ has affine fibers, the inertia $I_{f \circ p} \to F$ also has affine fibers because $p$ is affine. Hence, $f \circ p$ must be quasi-affine by Proposition \ref{PROP:quaff=genOX+affinestabilizers}.
But then for $J \subset I$ sufficiently large each $p_J \colon F_J \to Y$ must be already quasi-affine since $Y$ is quasi-compact \cite[Thm. C]{Ryd15:noetherianapprox}.
The morphism $p_J$ is a torsor for the relative product group $G \coloneq (\prod_{/Y})_{i \in J} \GL_{n_i,Y}$, where $n_i = \rk \shV_i$, and the classifying morphism $X \to \clsp G$ is quasi-affine because $p_J$ is. 

To finish the proof it suffices to construct a quasi-affine morphism $\clsp G \to \clsp \GL_{n,\bbZ}$.
The diagonal embedding $G \hookrightarrow \GL_{n,\bbZ}$, $n = \sum_{i \in J} \rk \shV_i$, induces a morphism of torsors and therefore a morphism $\clsp G \to \clsp \GL_{n,\bbZ}$, which is affine by smooth descent because the base change along the natural map $\Spec (\bbZ) \to \clsp \GL_{n,\bbZ}$ is the affine Stiefel scheme $\GL_{n,\bbZ}/G$.
\end{proof}

\begin{cor}
Let $f$ be a quasi-compact and quasi-separated morphism of algebraic stacks. If $f$ has the resolution property and the relative inertia $I_f \to X$ has affine fibers, then the diagonal $\Delta_f \colon X \to X \times_Y X$ is affine.
\end{cor}

\begin{proof}[Proof of Theorem \ref{THM:thmA}]
The implication \ref{THM:IT:thmA:basic} $\Rightarrow$ \ref{THM:IT:thmA:tengen}
follows from Theorem \ref{THM:tensor_generator}. Clearly, holds \ref{THM:IT:thmA:tengen} $\Rightarrow$ \ref{THM:IT:thmA:resprop}.
The remaining implication  \ref{THM:IT:thmA:resprop} $\Rightarrow$ \ref{THM:IT:thmA:basic} follows from Lemma \ref{LEM:resprop_closed_affines} and Theorem \ref{THM:resprop=basic} (with $Y = \Spec \bbZ$). 
Finally, the addendum was given in Proposition \ref{PROP:linred_tengen}, Corollary \ref{COR:tengen_char_zero} and Corollary \ref{COR:split_tensor_generator}.
\end{proof}

%%%%%%%%%%%%%%%%%%%%%%%%%%%%%%%%%%%%%%%%%%%%%%%%%%%%%%%%%%%%%%%%%%%%%%%%%%%%%%%%%%
%%%%%%%%%%%%%%%%%%%%%%%%%%%%%%%%%%%%%%%%%%%%%%%%%%%%%%%%%%%%%% DOCUMENT FOOTER %%% 
%%%%%%%%%%%%%%%%%%%%%%%%%%%%%%%%%%%%%%%%%%%%%%%%%%%%%%%%%%%%%%%%%%%%%%%%%%%%%%%%%%
\bibliographystyle{pgalpha}
\bibliography{pgbiblio}
\end{document}